\documentclass[12pt]{article}
\usepackage[a4paper,text={6.5in,9.5in},centering]{geometry} 
\usepackage[T1]{fontenc} 
\usepackage[utf8]{inputenc} 
\usepackage{amsmath} 
\usepackage{amssymb} 
\usepackage{amsfonts} 
\usepackage{amsthm} 
\usepackage{empheq} 
\usepackage{braket} 
\usepackage{esint} 
\usepackage{enumitem} 
\usepackage{url} 
\usepackage{mleftright} 
\usepackage{microtype} 
\mleftright

\theoremstyle{plain}
\newtheorem{theorem}{Theorem}[section]
\newtheorem{proposition}[theorem]{Proposition}
\newtheorem*{proposition*}{Proposition}
\newtheorem{lemma}[theorem]{Lemma}
\newtheorem{corollary}[theorem]{Corollary}
\newtheorem{definition}[theorem]{Definition}
\newtheorem*{definition*}{Definition}
\theoremstyle{remark}
\newtheorem{remark}[theorem]{Remark}

\newcommand{\dd}{\,\mathrm{d}}
\newcommand{\abs}[1]{\lvert #1 \rvert}
\newcommand{\norm}[1]{\lVert #1 \rVert}

\newcommand{\R}{\mathbf{R}}
\newcommand{\A}{\mathbb{A}}
\newcommand{\Gr}{\mathrm{Gr}}
\newcommand{\CC}[1]{\A #1:#1}
\renewcommand{\H}{\mathcal{H}}

\title{Uniform rectifiability of brittle fractures\\ in linear elasticity}

\author{C. Labourie}

\date{}

\begin{document}

\maketitle

\begin{abstract}
    We prove the uniform rectifiability of brittle fractures in arbitrary dimension. The existing approach for the Mumford-Shah functional, which relies on separation-type properties of the singular set, faces serious obstacles in the Griffith setting due to the lack of coarea formula for the symmetric gradient. We present an alternative route to uniform rectifiability for free-discontinuity problems by proving that cracks have ``plenty of big projections''.


\end{abstract}


\bigskip


\smallskip


\section{Introduction}

Variational models of fracture describe the formation of cracks as the outcome of an energy minimization principle. Let
$\Omega \subset \R^N$ be a bounded open set representing the reference configuration of a linearly elastic body.  When
the body $\Omega$ is subjected to a boundary deformation, it may deform and store energy, or release this energy by
breaking. The equilibrium configuration is found by minimizing a functional that balances the bulk elastic energy with
the energy required to create a crack, namely
\begin{equation*}
    \int_{\Omega} f(x,e(u)) \dd{x} + \int_{J_u} \phi(x,u^+,u^-,\nu_u) \dd{\H^{N-1}},
\end{equation*}
where $e(u) = (\nabla u + \nabla u^T)/2$ denotes the symmetrized gradient of $u$, the set $J_u$ is the discontinuity set
of $u$, the values $u^{\pm}$ are the traces of $u$ on each side of $J_u$ and the vector $\nu_u$ is a unit normal to the
crack.

In the context of Griffith's brittle fracture, the energy required to produce a crack is proportional to the surface
measure of the crack, yielding the Griffith functional
\begin{equation*}
    \int_{\Omega} f(x,e(u)) \dd{x} + \H^{N-1}(J_u).
\end{equation*}
The existence of minimizers for such functionals is established in the space $GSBD(\Omega)$ of generalized special functions
of bounded deformations \cite{dalmaso, CC2020}. A noteworthy feature of this framework is that the rectifiability of
the crack is directly provided by the general theory of $GSBD$ functions.

Let us recall that a set $E \subset \R^N$ is said to be rectifiable of dimension $N-1$ if it can be covered by countably
many Lipschitz images of $\R^{N-1}$, except for a set of $\H^{N-1}$ measure zero. This is equivalently characterized by
the fact that $E$ blows up as a hyperplane at $\H^{N-1}$ almost every point. Rectifiability is a flexible and useful
notion in geometric measure theory, but it is purely qualitative, which is not effective for describing fractures. As
cracks arise as minimizers of a functional, one expects them to have scale-invariant quantitative properties, rather
than merely good asymptotic behavior at generic points. Uniform rectifiability is a quantitative form of rectifiability
which is more natural in this setting.

As a first definition, a uniformly rectifiable set is a closed, Ahlfors-regular set $E \subset \R^N$ contained in the
image of an $\omega$-regular mapping $z \colon \R^{N-1} \longrightarrow \R^{N+1}$, where $\omega$-regular mappings will be
defined below and $\R^N$ is identified with a subspace of $\R^{N+1}$.
\begin{definition*}
    We say that a closed set $E \subset \R^N$ is uniformly rectifiable if there is a constant $C \geq 1$ such that
    \begin{equation*}
        C^{-1} r^{N-1} \leq \H^{N-1}(E \cap B(x,r)) \leq C r^{N-1} \quad \text{for all $x \in E$, $0 < r < \mathrm{diam}(E)$}
    \end{equation*}
    and if there is a weight $\omega \in \mathcal{A}_1(\R^{N-1})$ and an $\omega$-regular parametrization $z \colon \R^{N-1} \to \R^{N+1}$
    such that $E \subset z(\R^{N-1})$.
\end{definition*}

The Muckenhoupt classes $\mathcal{A}_p$ arose in the 1970s and 1980s as the appropriate conditions for the weighted
boundedness of many classical operators in harmonic analysis. The condition $\omega \in \mathcal{A}_1(\R^{N-1})$ means that
$\omega$ is a positive function in $L^1_{\mathrm{loc}}(\R^{N-1})$ such that
\begin{equation*}
    \fint_{B} \omega(y) \dd{y} \leq C \operatorname{ess inf}_{y \in B} \omega(y)
    \quad \text{$\forall$ ball $B \subset \R^{N-1}$}.
\end{equation*}
A mapping $z \colon \R^{N-1} \to \R^{N+1}$ is said to be $\omega$-regular if there is $C \geq 1$ such that
\begin{equation*}
    \abs{z(x) - z(y)} \leq C \biggl(\int_{B((x+y)/2,\abs{x-y})} \omega(u) \dd{u}\biggr)^{1/(N-1)} \quad \forall x,y \in \R^{N-1}
\end{equation*}
and
\begin{equation*}
    \int_{\set{y \in \R^{N-1} | z(y) \in B(x,r)}} \omega(y) \dd{y} \leq C r^{N-1}
    \quad \text{$\forall$ ball $B(x,r) \subset \R^{N+1}$.} 
\end{equation*}
One can think of $z$ as being Lipschitz in a reasonable weighted space $(\R^{N-1}, \omega \dd{x})$. The difference with
plain rectifiability is that a uniformly rectifiable set $E$ is contained in a single surface (instead of countably
many) and we have a uniform control on $E$ at all scales and locations through the constant $C$. This prevents the set
from ever being too scattered for instance, but still allows cusps and self-intersections to some extent.

An alternative way of characterizing uniformly rectifiable sets is by ``big pieces of biLipschitz images'' (BPBI),
\cite[Theorem 1.57]{DS93}.
\begin{proposition*}
    A closed set $E \subset \R^N$ is uniformly rectifiable if and only if there exists a constant $C \geq 1$ such that
    \begin{equation*}
        C^{-1} r^{N-1} \leq \H^{N-1}(E \cap B(x,r)) \leq C r^{N-1} \quad \text{for all $x \in E$, $0 < r < \mathrm{diam}(E)$}
    \end{equation*}
    and if for all $x \in E$ and $0 < r < \mathrm{diam}(E)$, there exist a compact subset $A \subset \R^{N-1}$ and a
    mapping $f \colon A \to \R^N$ such that
    \begin{equation*}
        C^{-1} \abs{z - w} \leq \abs{f(z) - f(w)} \leq C \abs{z - w} \quad \forall z,w \in A
    \end{equation*}
    and
    \begin{equation*}
        \H^{N-1}(E \cap f(A) \cap B(x,r)) \geq C^{-1} r^{N-1}.
    \end{equation*}
\end{proposition*}

It is not difficult to see that this is a rectifiability property
\footnote{Decompose $E = F \cup G$ as the disjoint union of a rectifiable part $F$ and a purely unrectifiable part $G$.
    As Lipschitz images are rectifiable, BPBI implies that for all $x \in E$, $\liminf_{r \to 0} r^{1-N} \H^{N-1}(F \cap
B(x,r)) \geq C^{-1}$.  One deduces that $G = E \setminus F$ has zero $\H^{N-1}$ measure by standard density theorems.}.

There is also a variant of this property called ``big pieces of Lipschitz graphs'' (BPLG), where one replaces
biLipschitz images by Lipschitz graphs: for all $x \in E$ and $0 < r < \mathrm{diam}(E)$, there exists a $C$-Lipschitz
graph $\Gamma$ of dimension $N-1$ such that
\begin{equation*}
    \H^{N-1}(E \cap \Gamma \cap B(x,r)) \geq C^{-1} r^{N-1}. 
\end{equation*}
This condition is strictly stronger than uniform rectifiability, as shown by an example of Hrycak (the ``venetian
blind'', see \cite[Subsection 1.2]{Orponen}).

In the present paper, we shall prove that brittle fractures are locally contained in a uniformly rectifiable set with big
pieces of Lipschitz graphs.  We refer to Section \ref{section_definition} for the relevant notation, in particular the
elasticity tensor $\A$, and the definition of topological quasiminimizers (which involves a multiplicative constant $M$).
\begin{theorem}[Uniform rectifiability of brittle fractures]
    There exist constants $\varepsilon_0 > 0$ and $C \geq 1$ (which depend on $N$, $M$, $\A$) such that the following
    holds. If $(u,K)$ is a Griffith topological quasiminimizer with gauge $h$ in a domain $\Omega$, then for all $x \in
    K$ and all $r > 0$ with $B(x,2r) \subset \Omega$ and $h(2r) \leq \varepsilon_0$,
    \begin{equation*}
        \text{$K \cap B(x,r)$ is contained in a uniformly rectifiable set $E$ (with BPLG) of constant $C$.}
    \end{equation*}
\end{theorem}

The Griffith energy presents substantial new challenges compared to its scalar analogue, the Mumford-Shah functional. We
begin by briefly recalling how David and Semmes \cite{DS96,DS96bis,DavidBOOK} established uniform rectifiability for
Mumford-Shah minimizers, in order to highlight these difficulties.  Let $(u,K)$ be a local Mumford-Shah minimizer in a
domain $\Omega$ (without fidelity term, for simplicity).  An easy case is when $\nabla u = 0$ on $\Omega \setminus K$.
In this situation, $K$ induces a minimal partition of $\Omega$ and locally satisfies the so-called ``Condition B'': for
all $x \in K$ and $r > 0$ such that $B(x,2r) \subset \Omega$, 
\begin{gather*}
    \text{there exist two balls $B_1, B_2 \subset B(x,r)$ of radius $\geq C^{-1} r$}\\
    \text{which lie in distinct connected components of $\Omega \setminus K$.}
\end{gather*}
This is a well-known sufficient condition for uniform rectifiability. Heuristically, if $K$ were not rectifiable, it
would be too scattered to separate points properly. In general however, Mumford-Shah minimizers do not satisfy Condition
B because of crack-tips. Nevertheless, in a ball $B$ where the elastic energy $\int_B \abs{\nabla u}^2 \dd{x}$ is small
compared to the surface term $\H^{N-1}(K \cap B)$, one can expect that $K$ should be close to a minimal partition.
In such a ball, David and Semmes proved that $K$ satisfies Condition B up to ``filling the holes'' by a level set
$\set{x \in B \setminus K | u(x) = t}$ of small, controlled measure. Their argument relies on a clever use of the
coarea formula to control the size of a level set. Moreover, a Carleson estimate provides a quantitative bound showing
that this occurs for most balls centered on $K$.  Thus, the singular set $K$ still attempts to separate its complement at
most locations and scales, and this weaker form of separation proves sufficient to deduce uniform rectifiability. Adapting
this approach to the Griffith setting is highly challenging, since no counterpart of the coarea formula exists for the
symmetrized gradient.

In a series of papers \cite{Friedrich1, Friedrich2, Solombrino}, Friedrich established a piecewise Korn inequality which
provides fine quantitative information on $GSBD$ functions and their jump set. This powerful tool was used in
\cite{FLK3} to bypass the coarea formula when controlling the size of holes, thereby making it possible to adapt the
proof of David and Semmes to the Griffith setting. Unfortunately, the piecewise Korn inequality is, as of now, only
available in the plane.

Prior to the work of David and Semmes, a first quantitative rectifiability property was introduced for Mumford-Shah
minimizers in the plane by Dibos and Koepfler \cite{dibos,dibos2}, and later extended to all dimensions by Solimini
\cite{Solimini}; this is the so-called ``property of projections''. We state here a slightly weaker variant. We say that
a closed Ahlfors-regular set $E \subset \R^N$ has the property of projections if, for all $x \in E$ and all $0 < r <
\mathrm{diam}(E)$,
\begin{equation*}
    \int_{\Gr(N-1)} \H^{N-1}\bigl(p_V(E \cap B(x,r))\bigr) \dd{V} \geq C^{-1} r^{N-1},
\end{equation*}
where $\Gr(N-1)$ denotes the Grassmannian space of hyperplanes in $\R^N$ equipped with its canonical measure, and $p_V$ is the
orthogonal projection onto $V$. This property implies rectifiability by the Besicovitch-Federer projection theorem and,
as observed by Dibos and Koepfler, is stable under Hausdorff convergence.
It is conjectured that the property of projections implies uniform rectifiability (see \cite[Open question
24.33]{DavidBOOK} and \cite{DS93,DS93bis}); however, to the best of the author's knowledge, this remains a difficult
open problem. Two major recent contributions are due to Dabrowski \cite{dabrowski}, who proved the conjecture for
one-dimensional sets, and Orponen \cite{Orponen}, who showed that BPLG is equivalent to a stronger condition called
``plenty of big projections'' (PBP): for all $x \in E$ and $0 < r < \mathrm{diam}(E)$, there exists $V \in \Gr(N-1)$
such that for all $W \in \Gr(N-1)$ with $\mathrm{dist}(V,W) \leq C^{-1}$,
\begin{equation*}
    \H^{N-1}\bigl(p_W(E \cap B(x,r))\bigr) \geq C^{-1} r^{N-1}.
\end{equation*}
The main contribution of the present paper is to prove that cracks have plenty of big projections. This provides a new
approach to uniform rectifiability for free-discontinuity functionals, alternative to the 1993 work of David and Semmes.

We begin by proving a first rectifiability theorem for Griffith minimizers, namely the property of projections (Theorem
\ref{thm_projection}). As the excision method of Dibos--Koepfler \cite{dibos,dibos2} and Solimini \cite{Solimini} does
not extend readily to the symmetric gradient, this is the first challenge in our paper. Our argument relies on a
Federer-Fleming type projection estimate, yielding many rays that do not meet the crack and along which the directional
variation of $u$ can be controlled. While the property of projections does not yet give uniform rectifiability, it has
the key feature of being stable under Hausdorff limits. We use this in a contradiction argument showing that every ball
contains a smaller shifted ball where the crack has a small flatness (Proposition \ref{prop_flat}). In such a ball, we
can then apply the slicing theorem of \cite{LL3} to control the size of holes through a projection.

This line of approach produces the ``flat in many places'' lemma in a self-contained way, instead of relying on the
theory of uniformly rectifiable sets, and yields at once both the uniform concentration property and plenty of big
projections.

\section{Definitions}\label{section_definition}

Our working space is an open set $\Omega \subset \R^N$, where $N \geq 2$. We say that a constant is \emph{universal}
when it depends only on $N$. Given a set $A$, the notation $A \subset \subset \Omega$ means that $A$ is compactly
contained in $\Omega$. We define a \emph{rigid motion} as an affine map $a \colon \R^N \to \R^N$ of the form $a(x) = b +
Ax$, where $b \in \R^N$ and $A \in \R^{N \times N}$ is a skew-symmetric matrix. This notion arises in linear elasticity
from linearizing rotations: a rotation $R$ close to $\mathrm{id}$ has a first order approximation of the form
$\mathrm{id} + A$, where $A$ is skew-symmetric. Notice that a Sobolev map $u$ satisfies $e(u) = 0$ if and only if it is
a rigid motion; thus the symmetric gradient only controls $u$ up to a rigid motion.

\medskip

\noindent{\bf Elasticity tensor.}
Given two matrices $\xi, \eta \in \R^{N \times N}$, the notation $\xi : \eta$ denotes the Frobenius inner product of
$\xi$ and $\eta$,
\begin{equation*}
    \xi : \eta := \sum_{ij} \xi_{ij} \eta_{ij}
\end{equation*}
and $\abs{\xi} = \sqrt{\xi : \xi}$ the Frobenius norm.
We fix for the whole paper a fourth-order elastic tensor $\A$ for which there exists a constant $\lambda \geq 1$ such
that
\begin{equation*}
    \A(\xi - \xi^T) = 0 \quad \text{and} \quad \A \xi : \xi \geq \lambda \abs{\xi + \xi^T}^2
    \quad \text{for all $\xi \in \R^{N \times N}$}.
\end{equation*}

\medskip

\noindent{\bf Coral pairs.}
A \emph{pair} $(u,K)$ consists of a relatively closed subset $K \subset \Omega$ and a Sobolev function $u \in
W^{1,2}_{\mathrm{loc}}(\Omega \setminus K;\R^N)$. We say that a pair $(u,K)$ has a \emph{locally finite energy} if for
all balls $B \subset \subset \Omega$, 
\begin{equation*}
    \int_{B \setminus K} \abs{e(u)}^2 \dd{x} + \H^{N-1}(K \cap B) < +\infty,
\end{equation*}
where $e(u) = (\nabla u + \nabla u^T)/2$. We say that a relatively closed set $K \subset \Omega$ is \emph{coral} if for
all $x \in K$, for all $r > 0$,
\begin{equation*}
    \H^{N-1}(K \cap B(x,r)) > 0,
\end{equation*}
where $\H^{N-1}$ is the Hausdorff measure of dimension $N-1$. A pair $(u,K)$ is also called coral if $K$ is coral in the
above sense.

\medskip

\noindent{\bf Competitors.}
Let $(u,K)$ be a pair. Let $B$ be an open ball such that $B \subset \subset \Omega$. A \emph{competitor} of $(u,K)$ in
$B$ is a pair $(\widetilde{u},\widetilde{K})$ such that
\begin{equation*}
    \widetilde{K} \setminus B = K \setminus B
    \quad \text{and} \quad
    \widetilde{u} = u \quad \text{a.e. in} \quad \Omega \setminus \left(K \cup B\right).
\end{equation*}
A \emph{topological competitor} of $(u,K)$ in $B$ is a competitor $(\widetilde{u},\widetilde{K})$ in $B$ such that
\begin{equation*}
    \text{all points $x,y \in \Omega \setminus (K \cup B)$
    which are separated by $K$ are also separated by $\widetilde{K}$}.
\end{equation*}
This means that if $x,y \in \Omega \setminus (K \cup B)$ belong to different connected components of $\Omega \setminus
K$, they also belong to different connected components of $\Omega \setminus \widetilde{K}$. This notion was introduced
by Bonnet \cite{Bonnet} and arises naturally when one looks at the minimality properties of limits of minimizers, such
as blow-up limits for instance.

\medskip

\noindent{\bf Quasiminimizers.}
We are interested in quasiminimizers in the sense of David and Semmes \cite{DavidBOOK,DS20}, which are meant to represent minimizers of functionals with possibly highly irregular coefficients.
Two conditions of this type are of interest.
First, assume that $u \in GSBD(\Omega)$ is a local minimizer of a functional of the form
\begin{equation*}
    \int_{\Omega} f(x,e(u)) \dd{x} + \int_{J_u} \psi(x,u^+,u^-,\nu_u) \dd{\H^{N-1}},
\end{equation*}
where $\psi \colon \Omega \times \R^N \times \R^N \times \mathbf{S}^{N-1} \to \R_+$
and $f \colon \Omega \times \R^{N \times N}_{\mathrm{sym}} \to \R_+$ are Borel functions satisfying
\begin{equation*}
    \psi(x,a,b,\nu) = \psi(x,b,a,  - \nu), \quad M^{-1} \leq \psi \leq M
\end{equation*}
and
\begin{equation*}
    \big\lvert f(x,\xi) - \CC{\xi} \big\rvert \leq g(\abs{\xi}^2),
\end{equation*}
where $g \colon [0,+\infty) \to [0,+\infty)$ is a non-decreasing function such that $\int_{1}^{+\infty} s^{-2} g(s) \dd{s} < +\infty$.
Here the bulk density is a non-homogeneous perturbation of $\CC{\xi}$ by a lower-order term, while the surface term is
very general, and may account for anisotropy and non-degenerate cohesive properties.
The minimality of $u$ entails (\cite[proof of Theorem 2.7]{FLK2}) that for all balls $B(x,r) \subset \subset \Omega$ and all $GSBD$ competitors $v$ of $u$ in $B(x,r)$,
\begin{multline}\label{eq_weak_quasi}
    \int_{B(x,r)} \CC{e(u)} \dd{x} + M^{-1} \H^{N-1}\bigl(J_u \cap B(x,r)\bigr)\\
    \leq \int_{B(x,r)} \CC{e(v)} \dd{x} + M \H^{N-1}\bigl(J_v \cap B(x,r)\bigr) + h(r) r^{N-1},
\end{multline}
where $h \colon (0,+\infty) \to [0,+\infty]$ is a non-decreasing function such that $\lim_{r \to 0} h(r) = 0$.
For instance, $g(t) = M(1 + t^{q/2})$ with $0 \leq q < 2$ yields a gauge of the form $h(r) = C r^{1-q/2}$ for $r < 1$.

If instead $u \in GSBD(\Omega)$ is a local minimizer of a functional of the form
\begin{equation*}
    \int_{\Omega} f(x,e(u)) \dd{x} + \mu(J_u),
\end{equation*}
where $\mu$ is a measure such that $M^{-1} \H^{N-1} \leq \mu \leq M \H^{N-1}$, one can remove $\mu(J_u \cap J_v)$ on
both sides and find a stronger condition: for all balls $B(x,r)  \subset \subset \Omega$ and all $GSBD$ competitors $v$ of $u$ in $B(x,r)$,
\begin{multline}\label{eq_strong_quasi}
    \int_{B(x,r)} \CC{e(u)} \dd{x} + M^{-1} \H^{N-1}\bigl(J_u \setminus J_v\bigr)\\
    \leq \int_{B(x,r)} \CC{e(v)} \dd{x} + M \H^{N-1}\bigl(J_v \setminus J_u\bigr) + h(r) r^{N-1},
\end{multline}
where $h \colon (0,+\infty) \to [0,+\infty]$ is a non-decreasing function such that $\lim_{r \to 0} h(r) = 0$.

Let us emphasize that condition (\ref{eq_weak_quasi}) is very weak.
It is primarily effective with competitors that try to eliminate $J_u$ completely from a given ball.
On the other hand, condition (\ref{eq_strong_quasi}) allows more complex energy comparisons.
The present author expects that (\ref{eq_strong_quasi}) prevents components of $\Omega \setminus K$ from being isolated,
from accumulating at a point, and from having cusps (at least in low dimensions, see \cite{LabTepl} in the absence of an
elastic term).
In the present work, we adopt (\ref{eq_weak_quasi}) as our definition of quasiminimality, since our technique is flexible enough to deal with this class.
While the existence of Griffith minimizers is established in the space $GSBD$ \cite{CC2020}, we shall directly work with
classical pairs $(u,K)$, since the jump sets of $GSBD$ quasiminimizers are essentially closed \cite[Theorem 2.7]{FLK2}.
We shall also take into account quasiminimizers whose competitors satisfy a topological constraint.
We define a \emph{gauge} as a non-decreasing function $h \colon (0,+\infty) \to [0,+\infty]$ satisfying $\limsup_{r \to 0} h(r) < \infty$.
It need not vanish as $r \to 0$; for instance $h$ may be a small constant.

\begin{definition}
    Let $M \geq 1$ and let $h$ be a gauge. A \emph{Griffith topological $M$-quasiminimizer} with gauge $h$ in $\Omega$ is a coral pair $(u,K)$ with locally finite energy such that for all $x \in \Omega$, for all $r > 0$ with $\overline{B}(x,r) \subset \Omega$ and for all topological competitors $(v,F)$ of $(u,K)$ in $B(x,r)$, we have
    \begin{multline*}
        \int_{B(x,r) \setminus K} \CC{e(u)} \dd{x} + M^{-1} \H^{N-1}(K \cap B(x,r)) \\\leq \int_{B(x,r) \setminus F} \CC{e(v)} \dd{x} + M \H^{N-1}(F \cap B(x,r)) + h(r) r^{N-1}.
    \end{multline*}
\end{definition}

\begin{remark}\label{rmk_scaling}
    We recall that if a pair $(u,K)$ is a topological $M$-quasiminimizer in $B(x_0,r_0)$ with gauge $h$, then the pair $(u_0,K_0)$ defined by
    \begin{equation*}
        u_0(x) := r_0^{-1/2} u(x_0 + x r_0) \quad \text{and} \quad  K_0 := r_0^{-1} (K - x_0)
    \end{equation*}
    is a topological $M$-quasiminimizer in $B(0,1)$, with gauge $h_0(t) := h(r_0 t)$.
    We will always work with scale invariant quantities.
\end{remark}

\medskip

\noindent{\bf Flatness.}
Given a pair $(u,K)$ in a ball $B(x,r) \subset \R^N$, we define the \emph{flatness} $\beta_K(x,r)$ of $K$ in $B(x,r)$ by
\begin{equation*}
    \beta_K(x,r) := r^{-1} \inf_{P} \sup_{y \in K \cap B(x,r)} \mathrm{dist}(y,P),
\end{equation*}
where $P$ ranges among affine hyperplanes passing through $x$ (the infimum is always attained by compactness of the Grassmann space).
When there is no ambiguity, we simply write $\beta(x,r)$.
It is straightforward to verify that for all balls $B(y,t) \subset B(x,r)$,
\begin{equation*}
    \beta_K(y,t) \leq 2\Bigl(\frac{r}{t}\Bigr) \beta_K(x,r).
\end{equation*}

\medskip

\noindent{\bf Normalized elastic energy.}
Given a pair $(u,K)$ in a ball $B(x,r) \subset \R^N$, we define the \emph{normalized elastic energy} of $u$ in $B(x,r)$ by
\begin{equation*}
    \omega(x,r) := r^{1 - N} \int_{B(x,r) \setminus K} \abs{e(u)}^2 \dd{x}.
\end{equation*}
More generally, for $p \geq 1$, we set
\begin{equation*}
    \omega_p(x,r):= r^{1-2N/p} \left(\int_{B(x,r)\setminus K} \abs{e(u)}^p \dd{x}\right)^{\frac{2}{p}}.
\end{equation*}
The exponent of $r$ is chosen so that $\omega_p$ is invariant under rescaling; see Remark \ref{rmk_scaling}.
Note that $\omega_2 = \omega$, and for $p \in [1,2]$, we have $\omega_p \leq C \omega$ by Hölder's inequality.

\medskip

For the rest of the paper, we fix a constant $M \geq 1$.
Every quasiminimizer considered in the following will be understood to satisfy the definition with this choice of $M$, which we omit to recall for simplicity.
We recall two useful properties of topological quasiminimizers.
The first one is Ahlfors-regularity \cite[Theorem 2.3 (2)]{FLK2}.

\begin{proposition}[Ahlfors-regularity]\label{prop_AR}
    There exist constants $\varepsilon_0 > 0$ and $C \geq 1$ (depending on $N$, $M$, $\A$) such that the following holds.
    Let $(u,K)$ be a topological quasiminimizer with any gauge $h$ in $\Omega$.
    For all $x \in K$ and all $r > 0$ such that $B(x,r) \subset \Omega$ and $h(r) \leq \varepsilon_0$, we have
    \begin{equation}\label{eq_AF0}
        \int_{B(x,r)} \abs{e(u)}^2 \dd{x} + \H^{N-1}(K \cap B(x,r)) \leq C r^{N-1}
    \end{equation}
    and
    \begin{equation}\label{eq_AF1}
        \H^{N-1}(K \cap B(x,r)) \geq C^{-1} r^{N-1}.
    \end{equation}
\end{proposition}

Our next property is a Carleson estimate, which shows that at most scales and locations, the elastic energy is small compared to the surface energy.

\begin{lemma}[Carleson estimate]\label{lem_carleson}
    There exist a constant $\varepsilon_0 > 0$ (depending on $N$, $M$, $\A$) and, for every $p \in [1,2)$, a constant $C_p \geq 1$ (depending on $N$, $M$, $\A$, $p$) such that the following holds.
    Let $(u,K)$ be a topological quasiminimizer with any gauge $h$ in $\Omega$.
    For all $x \in K$ and $r > 0$ such that $B(x,2r) \subset \Omega$ and $h(2r) \leq \varepsilon_0$, we have
    \begin{equation}\label{eq_carleson_estimate}
        \int_{y \in K \cap B(x,r)} \int_0^r \omega_p(y,t) \frac{\dd{t}}{t} \dd{\H^{N-1}(y)} \leq C_p r^{N-1}.
    \end{equation}
\end{lemma}
\begin{proof}
    The argument for (\ref{eq_carleson_estimate}) is the same as in Theorem 23.8 in \cite{DavidBOOK}, with $\abs{\nabla u}^2$ replaced by $\abs{e(u)}^2$.
    It only uses the fact that, if $h(2r) \leq \varepsilon_0$ with $\varepsilon_0$ smaller than in Proposition
    \ref{prop_AR}, the energy estimates (\ref{eq_AF0}) and (\ref{eq_AF1}) hold in every ball $B \subset B(x,2r)$
    centered on $K$.
\end{proof}

\begin{corollary}\label{cor_carleson}
    There exist a constant $\varepsilon_0 > 0$ (depending on $N$, $M$, $\A$) and, for all $p \in [1,2)$ and $\varepsilon > 0$, a constant $C(\varepsilon) \geq 1$ (depending on $N$, $M$, $\A$, $p$, $\varepsilon$) such that the following holds.
    Let $(u,K)$ be a topological quasiminimizer with any gauge $h$ in $\Omega$.
    For all $x \in K$ and $r > 0$ with $B(x,r) \subset \Omega$ and $h(r) \leq \varepsilon_0$, there exist $y \in K \cap B(x,r/2)$ and $t \in [C(\varepsilon)^{-1}r,r/2]$ such that $$\omega_p(y,t) \leq \varepsilon.$$
\end{corollary}
\begin{proof}
    This corollary is standard; we reproduce the proof for completeness.
    Fix $\varepsilon_0$ smaller than in Proposition \ref{prop_AR} and Lemma \ref{lem_carleson}.
    Let $x \in K$ and $r > 0$ be such that $B(x,r) \subset \Omega$ and $h(r) \leq \varepsilon_0$.
    Assume that there exists some constant $C_0 \geq 2$ such that $\omega_p(y,t) > \varepsilon$ for every $y \in K \cap B(x,r/2)$ and every $t \in [C_0^{-1}r,r/2]$.
    By Lemma \ref{lem_carleson} applied to $B(x,r/2)$, we have
    \begin{equation*}
        \int_{y \in K \cap B(x,r/2)} \int_0^{r/2} \omega_p(y,t) \frac{\dd{t}}{t} \dd{\H^{N-1}(y)} \leq C_p r^{N-1}.
    \end{equation*}
    On the other hand, by assumption on $C_0$,
    \begin{align*}
        \int_{y \in K \cap B(x,r/2)} \int_0^{r/2} \omega_p(y,t) \frac{\dd{t}}{t} \dd{\H^{N-1}(y)} &\geq \int_{y \in K \cap B(x,r/2)} \int_{C_0^{-1}r}^{r/2} \omega_p(y,t) \frac{\dd{t}}{t} \dd{\H^{N-1}(y)}\\
                                                                                                  &\geq \varepsilon \H^{N-1}(K \cap B(x,r/2)) \ln(C_0/2),
    \end{align*}
    which yields an upper bound on $C_0$.
    Taking $C_0$ large enough therefore ensures the existence of a pair $(y,t)$ with $\omega_p(y,t) \leq \varepsilon$.
\end{proof}

\section{The property of projections}

\begin{theorem}[Property of projections]\label{thm_projection}
    There exist constants $\varepsilon_0 > 0$ and $C \geq 1$ (depending on $N$, $M$, $\A$) such that the following holds.
    Let $(u,K)$ be a Griffith topological quasiminimizer in $\Omega$.
    For all $x \in K$ and all $r > 0$ such that $B(x,r) \subset \Omega$ and $h(r) \leq \varepsilon_0$, we have
    \begin{equation*}
        \int_{\Gr(N-1)} \H^{N-1}\bigl(p_V(K \cap B(x,r))\bigr) \dd{V} \geq C^{-1} r^{N-1}.
    \end{equation*}
\end{theorem}
\begin{proof}
    Choose an exponent $p \in (2(N-1)/N,2)$, for instance $p = (2N-1)/N$.
    To prove Theorem \ref{thm_projection}, it suffices to show that there exist constants $\varepsilon_1 > 0$ and $C \geq 1$ (depending on $N$, $M$, $\A$) such that for all $x \in K$ and $r > 0$ with $B(x,r) \subset \Omega$, $h(r) \leq \varepsilon_1$ and $\omega_p(x,r) \leq \varepsilon_1$, we have
    \begin{equation}\label{eq_projection_estimate}
        \int_{\Gr(N-1)} \H^{N-1}\bigl(p_V(K \cap B(x,r))\bigr) \dd{V} \geq C^{-1} r^{N-1}.
    \end{equation}
    Indeed, assume that such constants $\varepsilon_1$ and $C$ have been fixed, and let $\varepsilon_0$ be the constant of Corollary \ref{cor_carleson}.
    Let any $x \in K$ and $r > 0$ be such that $B(x,r) \subset \Omega$ and $h(r) \leq \min(\varepsilon_1,\varepsilon_0)$.
    By Corollary \ref{cor_carleson}, there exists a point $y \in K \cap B(x,r/2)$ and a radius $t$ with
    $C(\varepsilon_1)^{-1} r \leq t \leq r/2$ and $\omega_p(y,t) \leq \varepsilon_1$, where $C(\varepsilon_1) \geq 2$
    depends on $N$, $M$, $\A$.
    By definition of $\varepsilon_1$, the estimate (\ref{eq_projection_estimate}) then holds in the ball $B(y,t)$:
    \begin{equation*}
        \int_{\Gr(N-1)} \H^{N-1}\bigl(p_V(K \cap B(y,t))\bigr) \dd{V} \geq C^{-1} t^{N-1}.
    \end{equation*}
    Since $t \geq C(\varepsilon_1)^{-1} r$ and $K \cap B(y,t) \subset K \cap B(x,r)$, this yields the property of
    projections in $B(x,r)$, with a larger constant $C$.

    We now show that (\ref{eq_projection_estimate}) holds when $\omega_p$ is sufficiently small.
    After rescaling to the unit ball, we argue by contradiction and assume that there exist $\varepsilon_0, \varepsilon_1 \in (0,1)$ and a topological quasiminimizer $(u,K)$ in $B(0,1)$ such that
    \begin{equation}\label{eq_HU}
        \biggl(\int_{B(0,1)} \abs{e(u)}^p \dd{x}\biggr)^{1/p} \leq \varepsilon_0, \quad h(1) \leq \varepsilon_0 
    \end{equation}
    and
    \begin{equation}\label{eq_HK}
        \int_{\Gr(N-1)} \H^{N-1}\bigl(p_V(K \cap B(0,1))\bigr) \dd{V} \leq \varepsilon_1.
    \end{equation}
    We shall obtain a contradiction for a suitable choice of $\varepsilon_0 ,\varepsilon_1$ which depend only on $N$, $M$, $\A$ (we can already assume $\varepsilon_0$ to be less than in Proposition \ref{prop_AR}).
    Throughout the proof, $C$ denotes a generic constant $\geq 1$ which depends only on $N$, $M$, $\A$.

    By Ahlfors-regularity of $K$, it is standard (see \cite[Lemma 23.25]{DavidBOOK}) that
    \begin{equation}\label{eq_DS}
        \abs{\set{x \in B(0,1/2) | \mathrm{dist}(x,K) \leq \rho}} \leq C \rho \quad \forall \rho > 0.
    \end{equation}
    Choose a radius $r_0 \in (0,1/100)$ (depending on $N$, $M$, $\A$) so small that, for $\rho = 2 r_0$, the left-hand side has a volume strictly less than $\abs{B(0,1/100)}$.
    Hence, there exists a point $x_0 \in B(0,1/100)$ with $\mathrm{dist}(x_0,K) \geq 2 r_0$.
    Note in particular that $B(x_0,2r_0) \subset B(0,1/10)$.

    We shall also use the following observation.
    Let $(e_1,\ldots,e_{2N})$ denote the family of vectors such that $(e_1,\ldots,e_N)$ is the canonical basis of $\R^N$ and $(e_{N+1},\ldots,e_{2N}) = -(e_1,\ldots,e_N)$.
    The family of points $(x_0 + r_0 e_k)_k$ is well distributed along $\partial B(x_0,r_0)$ in the following sense:
    for every family of points $x_1,\ldots,x_{2N} \in \R^N$ such that $x_k \in B(x_0 + r_0 e_k, r_1)$ for all $k$, where $r_1 = r_0/(2N)$, we have
    \begin{equation}\label{eq_observation}
        \sup_{k} \abs{y \cdot (x - x_k)} \geq \frac{r_0}{2N} \abs{y} \quad \forall x,y \in \R^N.
    \end{equation}
    Indeed, for every $k$,
    \begin{equation*}
        \abs{y \cdot (x - x_k)} \geq \abs{y \cdot (x - x_0 - r_0 e_k)} - \frac{r_0}{2N} \abs{y}
    \end{equation*}
    and
    \begin{align*}
        \sup_{k} \abs{y \cdot (x - x_0 - r_0 e_k)} &\geq \frac{1}{2N}\sum_{k=1}^{2N} \abs{y \cdot (x - x_0 - r_0 e_k)}\\
                                                   &\geq \frac{1}{2N} \sum_{k=1}^{N} \abs{y \cdot (x - x_0 - r_0 e_k)} + \abs{y \cdot (x - x_0 + r_0 e_k)}\\
                                                   &\geq \frac{1}{2N} \sum_{k=1}^N \abs{y \cdot (x - x_0 - r_0 e_k) - y \cdot (x - x_0 + r_0 e_k)} = \frac{r_0}{N} \abs{y}.
    \end{align*}

    We now record three estimates that hold in average.
    Our first one is the Korn-Poincaré inequality applied to $u \in W^{1,2}_{\mathrm{loc}}(B(x_0,2 r_0);\R^N)$: there exist a universal constant $C_0 \geq 1$ and rigid motion $a$ such that
    \begin{equation*}
        \int_{B(x_0,2r_0)} \abs{u - a}^p \dd{x} \leq C_0 r_0^p \int_{B(x_0,2r_0)} \abs{e(u)}^p \dd{x}.
    \end{equation*}
    We may assume without loss of generality that $a = 0$ since a (quasi-)minimizer remains a (quasi-)minimizer after subtracting a rigid motion.
    Thus, we will only remember that
    \begin{equation}\label{eq_Korn}
        \int_{B(x_0,2r_0)} \abs{u}^p \dd{x} \leq C_0 \int_{B(0,1)} \abs{e(u)}^p \dd{x}.
    \end{equation}
    (Because of the gauge $h$ in the definition of quasiminimizers, we cannot use elliptic regularity to control
    directly the $L^{\infty}$ norm of $u$ in $B(x_0,r_0)$.)

    Next, we claim that, up to enlarging the universal constant $C_0$,
    \begin{equation}\label{eq_S}
        \int_{B(0,1)} \vartheta(x) \dd{x} \leq C_0 \int_{B(0,1)} \abs{e(u)}^p \dd{x},
    \end{equation}
    where,
    \begin{equation*}
        \vartheta(x) := \int_{z \in \partial B(0,1)} \biggl(\int_{(x + \R^+ z) \cap B(0,1)} \abs{e(u)}^p \dd{\H^1}\biggr) \dd{\H^{N-1}(z)}
    \end{equation*}
    Let us prove (\ref{eq_S}). For $x \in B(0,1)$ and $z \in \partial B(0,1)$, we have
    \begin{align*}
        \int_{(x + \R^+z) \cap B(0,1)} \abs{e(u)}^p \dd{\H^1} &= \int_0^{+\infty} \abs{e(u)(x + tz)}^p \mathbf{1}_{\set{\abs{x + tz} < 1}} \dd{t},
    \end{align*}
    so by Fubini, the change of variable $y = x + t z$ and the coarea formula in polar coordinates,
    \begin{align*}
        \vartheta(x)
        &\leq C \int_0^{+\infty} \frac{1}{t^{N-1}} \int_{\partial B(x,t) \cap B(0,1)} \abs{e(u)}^p \dd{\H^{N-1}(y)} \dd{t}\\
        &\leq C \int_{B(0,1)} \frac{\abs{e(u)}^p}{\abs{y - x}^{N-1}} \dd{y}.
    \end{align*}
    Applying Fubini once more
    \begin{equation*}
        \int_{B(0,1)} \vartheta(x) \dd{x} \leq C \int_{B(0,1)} \abs{e(u)}^p \biggl(\int_{B(0,1)} \frac{\dd{x}}{\abs{y - x}^{N-1}}\biggr) \dd{y}
    \end{equation*}
    and the result follows by integrability of $x \mapsto \abs{x}^{1-N}$ near the origin.

    We now turn to our last estimate. We claim that up to increasing the universal constant $C_0$, we have
    \begin{equation}\label{eq_FF}
        \int_{B(0,1)} \H^{N-1}\bigl(\Sigma(x)\bigr) \dd{x} \leq C_0 \int_{\Gr(N-1)} \H^{N-1}\bigl(p_V(K \cap B(0,1))\bigr) \dd{V},
    \end{equation}
    where, $\Gr(k)$ denotes the Grassmannian of $k$-linear subspaces in $\R^N$ and
    \begin{equation*}
        \Sigma(x) := \set{z \in \partial B(0,1) | (x + \R^+ z) \cap K \cap B(0,1) \ne \emptyset}.
    \end{equation*}
    For $k = 1,\ldots,N-1$, let $\gamma_{k}$ denote the canonical measure of $\Gr(k)$ (see \cite[Chapter 3, \S 3.9]{Mattila});
    we will actually only use $\gamma_{1}$ and $\gamma_{N-1}$.
    We recall that for all Borel sets $A \subset \Gr(1)$,
    \begin{equation*}
        \gamma_{1}(A) = \gamma_{N-1}\bigl(\set{L^\perp | L \in A}\bigr),
    \end{equation*}
    and that for all Borel sets $S \subset \partial B(0,1)$,
    \begin{equation}\label{eq_SG2}
        \H^{N-1}(S) \leq C \gamma_{1}\bigl(\set{L | L \cap S \ne \emptyset}\bigr).
    \end{equation}
    Applying (\ref{eq_SG2}) with $S = \Sigma(x)$, we obtain that for all $x \in B(0,1)$,
    \begin{equation*}
        \H^{N-1}(\Sigma(x)) \leq C \gamma_{1}\bigl(\set{L | \left(x + L\right) \cap K \cap B(0,1) \ne \emptyset}\bigr).
    \end{equation*}
    Then, by Fubini's theorem,
    \begin{align*}
        \int_{B(0,1)} \H^{N-1}(\Sigma(x)) \dd{x} &\leq C \int_{\Gr(1)} \abs{\set{x \in B(0,1) | (x + L) \cap K \cap B(0,1)\ne \emptyset}} \dd{L}\\
                                                 &\leq C \int_{\Gr(N-1)} \abs{\set{x \in B(0,1) | (x + V^\perp) \cap K \cap B(0,1) \ne \emptyset}} \dd{V}\\
                                                 &\leq C \int_{\Gr(N-1)} \H^{N-1}\bigl(p_V(K \cap B(0,1))\bigr) \dd{V},
    \end{align*}
    which is exactly (\ref{eq_FF}).

    Now, we deduce that $\abs{u(x)}^p$, $\vartheta(x)$ and $\H^{N-1}(\Sigma(x))$ are controlled at most points $x$.
    Let $C_1 > 0$ (to be fixed shortly), and let $A$ be the set of points $x \in B(x_0,2r_0)$ such that
    \begin{equation}\label{eq_badKorn}
        \abs{u(x)}^p > C_1 \int_{B(0,1)} \abs{e(u)}^p \dd{x}.
    \end{equation}
    If $\int_{B(0,1)} \abs{e(u)}^p \dd{x} > 0$, we use (\ref{eq_Korn}) and Chebyshev's inequality,
    \begin{equation*}
        \abs{A} \biggl(C_1 \int_{B(0,1)} \abs{e(u)}^p \dd{x}\biggr) \leq \int_{B(x_0,2r_0)} \abs{u(x)}^p \dd{x} \leq C_0 \int_{B(0,1)} \abs{e(u)}^p \dd{x},
    \end{equation*}
    so $\abs{A} \leq C_0 C_1^{-1}$.
    If $\int_{B(0,1)} \abs{e(u)}^p \dd{x} = 0$, then (\ref{eq_Korn}) shows that $\abs{u} = 0$ a.e. so that we have $\abs{A} = 0$ anyway.
    Similarly, the set of points $x \in B(0,1/10)$ such that 
    \begin{equation}\label{eq_badS}
        \vartheta(x) > C_1 \int_{B(0,1)} \abs{e(u)}^p \dd{x}
    \end{equation}
    has measure at most $C_1^{-1} C_0$ by (\ref{eq_S}), and the set of points $x \in B(0,1/10)$ such that
    \begin{equation}\label{eq_badFF}
        \H^{N-1}(\Sigma(x)) > C_1 \int_{\Gr(N-1)} \H^{N-1}\bigl(p_V(K \cap B(0,1))\bigr) \dd{V},
    \end{equation}
    also has measure at most $C_1^{-1} C_0$.
    We may therefore choose $C_1$ so large (depending on $N$, $M$, $\A$) that the bad set where either
    (\ref{eq_badKorn}), (\ref{eq_badS}) or (\ref{eq_badFF}) holds has a volume strictly smaller than
    that of any ball of the form $B(x_0 + r_0 e_k, r_1)$, where we recall $r_1 = r_0/(2N)$.
    Consequently, for each $k = 1,\ldots,2N$, we can find a point $x_k \in B(x_0 + r_0 e_k, r_1)$ such that
    \begin{equation}\label{eq_xKorn}
        \abs{u(x_k)}^p \leq C_1 \int_{B(0,1)} \abs{e(u)}^p \dd{x} , 
    \end{equation}
    \begin{equation}\label{eq_xS}
        \vartheta(x_k)
        \leq C_1 \int_{B(0,1)} \abs{e(u)}^p \dd{x},
    \end{equation}
    and
    \begin{equation}\label{eq_xFF}
        \H^{N-1}(\Sigma(x_k)) \leq C_1 \int_{\Gr(N-1)} \H^{N-1}\bigl(p_V(K \cap B(0,1))\bigr) \dd{V}.
    \end{equation}
    From this point on, the constant $C_1$ is fixed and depends on $N$, $M$, $\A$, so rename it $C$ as usual.

    Given a constant $C_2 > 0$, we define for each $k$, the set $\Sigma'(x_k)$ of points $z \in \partial B(0,1)$ such that
    \begin{equation*}
        \int_{(x_k + \R^+ z) \cap B(0,1)} \abs{e(u)}^p \dd{x} > C_2 \int_{B(0,1)} \abs{e(u)}^p \dd{x}.
    \end{equation*}
    By (\ref{eq_xS}) and a now usual Chebyshev argument, we obtain $\H^{N-1}(\Sigma'(x_k)) \leq C C_2^{-1}$.
    We now choose $C_2 = C(\varepsilon_1)$ so large that 
    \begin{equation}\label{eq_xFF2}
        \H^{N-1}(\Sigma'(x_k)) \leq \varepsilon_1,
    \end{equation}
    where $\varepsilon_1$ was defined near (\ref{eq_HK}). 
    In the rest of the proof, $C(\varepsilon_1)$ stands for a generic constant $\geq 1$ which depends on $N$, $M$, $\A$ and $\varepsilon_1$.

    We finally introduce for $k = 1,\ldots,2N$, the union of (bad) rays
    \begin{equation*}
        R_k = B(0,1) \cap \bigcup \set{x_k + \R^+ z | z \in \Sigma(x_k) \cup \Sigma'(x_k)}
    \end{equation*}
    and the (bad) set $R = \bigcup_{k=1}^{2N} R_k$.
    The point is that for $x \in B(0,1) \setminus R$, the segment $[x_k,x]$ does not meet $K$ and
    \begin{equation}\label{eq_C2}
        \int_{[x_k,x]} \abs{e(u)}^p \dd{\H^1} \leq C(\varepsilon_1) \int_{B(0,1)} \abs{e(u)}^p \dd{x}.
    \end{equation}
    This property will play a key role in controlling the elastic energy of our competitor, which we construct in the next step.

    By (\ref{eq_xFF}) and (\ref{eq_HK}), we have $\H^{N-1}(\Sigma(x_k)) \leq C \varepsilon_1$, and by (\ref{eq_xFF2})
    we know that $\H^{N-1}(\Sigma'(x_k)) \leq \varepsilon_1$. It follows that for all $k$, the set
    \begin{equation*}
        V_k = \set{tz | z \in \Sigma(x_k) \cup \Sigma'(x_k),\; 0 \leq t \leq 2}
    \end{equation*}
    has a volume $\leq C \varepsilon_1$ and in turn $\abs{R} \leq C \varepsilon_1$ since for each $k$, $R_k \subset x_k + V_k$.

    Let $D$ denote the annulus
    \begin{equation*}
        D = \set{x \in \R^N | 1/4 < \abs{x} < 1/2}.
    \end{equation*}
    We claim that for $\delta \in (0,1/10)$,
    \begin{equation}\label{eq_controlu}
        \int_{\set{x \in D | \mathrm{dist}(x,K) \geq \delta}} \abs{u}^p \dd{x} \leq C \delta^{p-N} C(\varepsilon_1) \varepsilon_0^p.
    \end{equation}
    Choose a maximal family of points $y_i \in D$ such that $\mathrm{dist}(y_i,K) \geq \delta$ and $\abs{y_i - y_j} \geq \delta/10$.
    Then the balls $B(y_i,\delta/100)$ are pairwise disjoint, and since they are contained in $B(0,1)$, their number is at most $C \delta^{-N}$.
    Moreover, the set $\set{x \in D | \mathrm{dist}(x,K) \geq \delta}$ is covered by the balls $B_i := B(y_i,\delta/10)$,
    so to prove (\ref{eq_controlu}), it suffices to show that for all $i$,
    \begin{equation}\label{eq_controluBi}
        \int_{B_i} \abs{u}^p \dd{x} \leq C \delta^p C(\varepsilon_1) \varepsilon_0^p.
    \end{equation}
    Since $B_i \cap K = \emptyset$, we can apply Korn-Poincaré inequality in $B_i$ and obtain a rigid motion $a_i$ such that
    \begin{equation}\label{eq_iKorn}
        \int_{B_i} \abs{u - a_i}^p \dd{x} \leq C \delta^p \int_{B_i} \abs{e(u)}^p \dd{x}. 
    \end{equation}
    We now show that for $x \in B_i \setminus R$, we have $\abs{u(x)} \leq C(\varepsilon_1) \varepsilon_0$.
    By definition of $R$, together with (\ref{eq_HU}), we know that for every such $x$ and for all $k$, the segment $[x_k,x]$ does not intersect $K$ and
    \begin{equation*}
        \int_{[x_k,x]} \abs{e(u)}^p \dd{\H^1} \leq C(\varepsilon_1) \int_{B(0,1)} \abs{e(u)}^p \dd{x} \leq C(\varepsilon_1) \varepsilon_0^p.
    \end{equation*}
    Since
    \begin{equation*}
        \frac{\dd}{\dd{t}}\bigl[u(x_k + t(x-x_k)) \cdot (x - x_k)\bigr] = \bigl[e(u)(x_k + t(x - x_k)) (x - x_k)\bigr] \cdot (x - x_k), 
    \end{equation*}
    we find
    \begin{equation*}
        \abs{(u(x) - u(x_k)) \cdot (x - x_k)} \leq C \abs{x - x_k} \int_{[x_k,x]} \abs{e(u)} \dd{\H^1} \leq C(\varepsilon_1) \varepsilon_0.
    \end{equation*}
    From (\ref{eq_xKorn}), this implies $\abs{u(x) \cdot (x - x_k)} \leq C(\varepsilon_1) \varepsilon_0$ for all $k$,
    and then (\ref{eq_observation}) entails $\abs{u(x)} \leq C(\varepsilon_1) \varepsilon_0$.
    As this estimate holds for all $x \in B_i \setminus R$ and as $\abs{B_i} \leq C \delta^N \leq C \delta^p$,
    we deduce that
    \begin{equation*}
        \int_{B_i \setminus R} \abs{a_i}^p \dd{x} \leq C \int_{B_i \setminus R} \abs{u - a_i}^p \dd{x} + C \int_{B_i \setminus R} \abs{u}^p \dd{x} \leq \delta^p C(\varepsilon_1) \varepsilon_0^p.
    \end{equation*}
    As $\abs{R} \leq C \varepsilon_1$, we may choose $\varepsilon_1$ sufficiently small (depending on $N$, $\delta$) so that $\abs{B_i \setminus R} \geq \abs{B_i}/2$.
    Lemma \ref{lem_rigid} in the appendix then yields 
    \begin{equation*}
        \abs{a_i}^p \leq \delta^{p-N} C(\varepsilon_1) \varepsilon_0^p \quad \text{uniformly on $B_i$.}
    \end{equation*}
    Inserting this bound into (\ref{eq_iKorn}) gives (\ref{eq_controluBi}), and in turn (\ref{eq_controlu}).

    Consider a cut-off function $\varphi \in C^1(D;[0,1])$ such that $\varphi(x) = 1$ when $\mathrm{dist}(x,K) \geq 2 \delta$, $\varphi(x) = 0$ when $\mathrm{dist}(x,K) \leq \delta$ and $\abs{\nabla \varphi} \leq C \delta^{-1}$ everywhere.
    We extend $u$ on the whole annulus $D$ by $v(x) := \varphi(x) u(x)$.
    One estimates
    \begin{equation*}
        \abs{e(v)} \leq \abs{\nabla \varphi} \abs{u} + \varphi \abs{e(u)},
    \end{equation*}
    so
    \begin{align*}
        \int_D \abs{e(v)}^p \dd{x} &\leq C \delta^{-p} \int_{\set{x \in D | \mathrm{dist}(x,K) \geq \delta}} \abs{u}^p  \dd{x} + C \int_D \abs{e(u)}^p \dd{x}\\
                                   &\leq C \delta^{-N} C(\varepsilon_1) \varepsilon_0^p.
    \end{align*}
    As $v$ is a Sobolev map in the annulus $D$, Korn's inequality provides a rigid motion $a$ such that
    \begin{equation*}
        \int_{D} \abs{\nabla v - \nabla a}^p \dd{x} \leq C \delta^{-N} C(\varepsilon_1) \varepsilon_0^p.
    \end{equation*}
    By averaging over radii, we deduce that there exists $\rho \in (1/4,1/2)$ such that
    \begin{equation*}
        \int_{\partial B(0,\rho)} \abs{\nabla v - \nabla a}^p \dd{\H^{N-1}} \leq C \delta^{-N} C(\varepsilon_1) \varepsilon_0^p.
    \end{equation*}
    We may further assume that the trace of $v - a$ on $\partial B(0,\rho)$ belongs to $W^{1,p}(\partial B(0,\rho))$,
    with gradient given by the tangential component of $\nabla v - \nabla a$.
    Moreover, applying again the coarea formula to (\ref{eq_DS}), we may additionally choose $\rho \in (1/4,1/2)$ so that
    \begin{equation*}
        Z := \set{x \in \partial B(0,\rho) | \mathrm{dist}(x,K) \leq 2 \delta}
    \end{equation*}
    satisfies $\H^{N-1}(Z) \leq C \delta$.
    By \cite[Lemma 22.32]{DavidBOOK} (recall we choose $p$ such that $2(N-1)/N < p < 2$), the function $w := v - a$ defined on $\partial B(0,\rho)$ admits a $C^1$ extension to $B(0,\rho)$ such that
    \begin{equation*}
        \int_{B(0,\rho)} \abs{\nabla w}^2 \dd{x} \leq C \biggl(\int_{\partial B(0,\rho)} \abs{\nabla v - \nabla a}^p \dd{x}\biggr)^{2/p} \leq C \delta^{-2N/p} C(\varepsilon_1) \varepsilon_0^2.
    \end{equation*}

    Assume for the moment that $(u,K)$ is a plain quasiminimizer, we will treat later the topological case.
    Define $(\widetilde{u},\widetilde{K})$ by
    \begin{equation*}
        \widetilde{K} = \bigl(K \setminus B(0,\rho)\bigr) \cup Z
    \end{equation*}
    and
    \begin{equation*}
        \widetilde{u} =
        \begin{cases}
            u   &\; \text{in $B(0,1) \setminus \bigl(B(0,\rho) \cup K\bigr)$}\\
            w+a &\; \text{in $\overline{B}(0,\rho) \setminus Z$}.
        \end{cases}
    \end{equation*}
    Then $(\widetilde{u},\widetilde{K})$ is a competitor of $(u,K)$ in every ball $B(0,r)$ with $\rho < r < 1$, and hence
    \begin{multline*}
        \int_{B(0,\rho)} \CC{e(u)} \dd{x} + M^{-1} \H^{N-1}(K \cap B(0,\rho))\\
        \leq \int_{B(0,\rho)} \CC{e(w)} \dd{x} + M \H^{N-1}(Z) + h(1).
    \end{multline*}
    By Ahlfors-regularity (Proposition \ref{prop_AR}), there exists a constant $c_0 > 0$ which only depends on $N$, $M$, $\A$ such that
    \begin{equation*}
        \int_{B(0,\rho)} \CC{e(u)} \dd{x} + M^{-1} \H^{N-1}(K \cap B(0,\rho)) \geq c_0.
    \end{equation*}
    On the other hand, we can bound the right-hand side using
    \begin{equation*}
        \int_{B(0,\rho)} \abs{e(w)}^2 \dd{x} \leq C \delta^{-2N/p} C(\varepsilon_1) \varepsilon_0^2 \quad \text{and} \quad M \H^{N-1}(Z) \leq C M \delta.
    \end{equation*}
    We first fix $\delta > 0$ small enough (depending on $N$, $M$, $\A$) so that $M \H^{N-1}(Z) \leq c_0/10$,
    and then choose $\varepsilon_0$ sufficiently small in (\ref{eq_HU}) (again depending on $N$, $M$, $\A$)
    so that $\int_{B(0,\rho)} \abs{e(w)}^2 \dd{x} \leq c_0/10$ and $h(1) \leq c_0/10$.
    This contradicts the previous lower bound.

    In the case where $(u,K)$ is a topological quasiminimizer, choose any point $x_* \in B(0,1/10)$ such that the set
    \begin{equation*}
        \Sigma := \set{z \in \partial B(0,1) | (x_* + \R^+ z) \cap K \cap B(0,1) \ne \emptyset}
    \end{equation*}
    has $\H^{N-1}$-measure $\leq C \varepsilon_1$, for some universal constant $C \geq 1$ (such a point exists by
    (\ref{eq_FF})).
    Then the set
    \begin{equation*}
        S := \set{x \in \partial B(0,\rho) | [x,x_*] \cap K \ne \emptyset}
    \end{equation*}
    also has $\H^{N-1}$ measure $\leq C \varepsilon_1$ since it is contained in the inverse image of $\Sigma$ under the bilipschitz map
    \begin{equation*}
        \begin{aligned}
            \partial B(0,\rho) &\longrightarrow \partial B(0,1)\\
            x &\longmapsto \frac{x - x_*}{\abs{x - x_*}}.
        \end{aligned}
    \end{equation*}
    We then define a competitor $\widetilde{K}$ by
    \begin{equation*}
        \widetilde{K} = \bigl(K \setminus B(0,\rho)\bigr) \cup Z \cup S.
    \end{equation*}
    The set $\widetilde{K}$ is relatively closed in $\Omega$. Indeed $S$ is the image of $K \cap \overline{B}(0,\rho)$
    under the radial projection of center $x_*$ onto $\partial B(0,\rho)$, and is therefore compact.
    Moreover, $\H^{N-1}(S) \leq C \varepsilon_1$, so the contradiction argument above still applies provided we also assume $\varepsilon_1 \leq c_0/(10 C)$.

    We now check that $\widetilde{K}$ is a topological competitor of $K$.
    Let $x,y \in B(0,1) \setminus \overline{B}(0,\rho)$ be two points that are connected by a path $\gamma$ in $B(0,1) \setminus \widetilde{K}$.
    We claim that $x$ and $y$ are connected in $B(0,1) \setminus K$.
    If $\gamma$ never meets $\overline{B}(0,\rho)$, then $x$ and $y$ are also connected in $B(0,1) \setminus K$ because $\widetilde{K}= K$ outside $\overline{B}(0,\rho)$.
    Otherwise, let $x_1$ be the first point where $\gamma$ meets $\partial B(0,\rho)$, and $y_1$ be the last such point.
    Then $x$ is connected to $x_1$ by a path in $B(0,1) \setminus K$ and $y$ to $y_1$ by a path in $B(0,1) \setminus K$.
    Since $x_1, y_1 \notin S$, the segments $[x_*,x_1]$ and $[x_*,y_1]$ don't intersect $K$.
    Concatenating these four paths shows that $x$ and $y$ are connected in $B(0,1) \setminus K$, as required.
\end{proof}

The property of projections already ensures that the crack is rectifiable (assuming that $\limsup_{r \to 0} h(r)$ is small enough).
To see this, decompose as usual $K = F \cup G$ as the disjoint union of a rectifiable part $F$ and a purely unrectifiable part $G$.
According to the Besicovitch-Federer projection theorem, we have
\begin{equation*}
    \int_{\Gr(N-1)} \H^{N-1}\bigl(p_V(G)\bigr) \dd{V} = 0.
\end{equation*}
Therefore, Theorem \ref{thm_projection} shows that for all $x \in K$ and $r > 0$ such that $B(x,r) \subset \Omega$ and $h(r) \leq \varepsilon_0$,
\begin{equation*}
    C^{-1} r^{N-1} \leq \int_{\Gr(N-1)} \H^{N-1}\bigl(p_V(F \cap B(x,r))\bigr) \dd{V} \leq \H^{N-1}(F \cap B(x,r)).
\end{equation*}
One deduces that $\H^{N-1}(G) = 0$ by standard density theorems \cite[Theorem 6.2]{Mattila}.

The next result states that, given any ball, one can find a smaller shifted ball (with comparable radius) where $K$ has a small flatness.
This is a standard property of uniformly rectifiable sets, typically derived from the ``Weak Geometric Lemma'' (WGL); see \cite[Section 73]{DavidBOOK}.
Here we obtain it instead by a contradiction argument and the stability of the projection property under Hausdorff convergence.
This proof is inspired by works on the Plateau problem, such as those of Fang \cite{Fang,FangPHD} and David \cite[Lemma 10.21]{David2019}, where the rectifiability of limits is obtained by different methods but used in the same spirit.



\begin{proposition}[Flat in many places]\label{prop_flat}
    There exist a constant $\varepsilon_0 > 0$ (depending on $N$, $M$, $\A$) and, for all $\varepsilon > 0$, a constant $C(\varepsilon) \geq 1$ (depending on $N$, $M$, $\A$, $\varepsilon$) such that the following holds.
    If $(u,K)$ is a topological quasiminimizer in $\Omega$, then for all $x \in K$ and all $r > 0$ with $B(x,r) \subset \Omega$ and $h(r) \leq \varepsilon_0$, there exists $y \in K \cap B(x,r/2)$ and $t \in [C(\varepsilon)^{-1} r,r/2]$ such that
    \begin{equation*}
        \beta(y,t) \leq \varepsilon.
    \end{equation*}
\end{proposition}
\begin{proof}
    Fix $\varepsilon_0$ smaller than in Proposition \ref{prop_AR} and Theorem \ref{thm_projection}.
    Up to rescaling, it suffices to prove the statement for quasiminimizers in the unit ball.
    We proceed by contradiction for a given $\varepsilon > 0$.
    Suppose there exists a sequence of topological quasiminimizers $(u_i,K_i)_i$ with gauges $h_i$ in $B(0,1)$ such that,
    for all $i$, $h_i(1) \leq \varepsilon_0$ and, for all $y \in K_i \cap B(0,1/2)$, for all $t \in [2^{-i},1/2]$,
    one has $\beta_{K_i}(y,t) > \varepsilon$.

    We extract a subsequence (not relabeled) such that $(K_i)_i$ converges to a relatively closed subset $K \subset B(0,1)$ in local Hausdorff distance (see \cite[Proposition 34.6]{DavidBOOK}).
    Since each $K_i$ contains $0$, the limit set $K$ also contains $0$.
    It is standard that the limit set $K$ inherits Ahlfors-regularity from the sequence $(K_i)$; that is, there exists $C_0 \geq 1$ such that for all $x \in K$ and all $r > 0$ with $B(x,r) \subset B(0,1)$,
    \begin{equation*}
        C_0^{-1} r^{N-1} \leq \H^{N-1}(K \cap B(x,r)) \leq C_0 r^{N-1}.
    \end{equation*}
    Next, fix $x \in K \cap B(0,1)$ and $r > 0$ with $B(x,r) \subset B(0,1)$.
    We claim that $K$ has the property of projections in $B(x,r)$.
    Without loss of generality, we may assume that $\overline{B}(x,r) \subset B(0,1)$.
    Let $0 < \rho < r$ and $\eta \in (0,r-\rho)$.
    For all sufficiently large $i$, every point of $K_i \cap B(x,\rho)$ lies at distance at most $\eta$ from $K \cap \overline{B}(x,r)$; hence
    \begin{multline*}
        \int_{\Gr(N-1)} \H^{N-1}\bigl(p_V(K_i \cap B(x,\rho))\bigr) \dd{V} \\
        \leq \int_{\Gr(N-1)} \H^{N-1}\bigl(\set{z \in V | \mathrm{dist}\bigl(z,p_V(K \cap \overline{B}(x,r))\bigr) \leq \eta}\bigr) \dd{V}.
    \end{multline*}
    Since $x \in K$, there exists a sequence of points $(x_i)_i$ with $x_i \in K_i \cap B(0,1)$ and $x_i \to x$.
    In particular, for $i$ large enough we have
    \begin{equation*}
        \int_{\Gr(N-1)} \H^{N-1}\bigl(p_V(K_i \cap B(x_i,\rho/2))\bigr) \dd{V}
        \leq \int_{\Gr(N-1)} \H^{N-1}\bigl(p_V(K_i \cap B(x,\rho))\bigr) \dd{V}.
    \end{equation*}
    By applying Theorem \ref{thm_projection} to $(u_i,K_i)$ in the ball $B(x_i,\rho/2)$, we obtain
    \begin{equation*}
        \int_{\Gr(N-1)} \H^{N-1}\bigl(\set{z \in V | \mathrm{dist}\bigl(z,p_V(K \cap \overline{B}(x,r))\bigr) \leq \eta}\bigr) \dd{V} \geq C^{-1} \rho^{N-1}.
    \end{equation*}
    Letting $\rho \to r$ and $\eta \to 0$, and using dominated convergence, we deduce that
    \begin{equation*}
        \int_{\Gr(N-1)} \H^{N-1}\bigl(p_V(K \cap \overline{B}(x,r))\bigr) \dd{V} \geq C^{-1} r^{N-1}.
    \end{equation*}
    This proves the claim, and therefore $K$ is rectifiable.

    As a standard consequence from Ahlfors-regularity and rectifiability, we have
    \begin{equation*}
        \lim_{r \to 0} \beta(x,r) = 0 \quad \text{for} \quad \text{$\H^{N-1}$-a.e. $x \in K$.}
    \end{equation*}
    Since $0 \in K$ and $K$ is Ahlfors-regular, we also have $\H^{N-1}(K \cap B(0,1/10)) > 0$, so we can find a point $x \in K \cap B(0,1/10)$ such that $\lim_{r \to 0} \beta(x,r) = 0$.
    Then, there exists $r \in (0,1/10)$ and a hyperplane $P$ through $x$ such that $B(x,r) \subset B(0,1)$ and $$K \cap B(x,r) \subset \set{y \in B(x,r) | \mathrm{dist}(y,P) \leq \varepsilon r/200}.$$
    Let $(x_i) \in K_i$ be a sequence such that $x_i \to x$.
    For $i$ sufficiently large,
    \begin{equation*}
        K_i \cap B(x_i,r/10) \subset \set{y \in B(x_i,r/10) | \mathrm{dist}(y,P) \leq \varepsilon r/100},
    \end{equation*}
    which proves that $\beta_{K_i}(x_i,r/10) \leq \varepsilon/10$.
    Choosing $i$ large enough so that $2^{-i} < r/10$, this contradicts our assumption that $\beta_{K_i}(y,t) \geq \varepsilon$ for all $y \in K_i \cap B(0,1/2)$ and all $t \in [2^{-i},1/2]$.
\end{proof}

Combining the previous Proposition \ref{prop_flat} (flat in many places) with Corollary \ref{cor_carleson} (small elastic energy in many places), we obtain both conditions simultaneously.

\begin{corollary}\label{cor_nice}
    There exist a constant $\varepsilon_0 > 0$ (depending on $N$, $M$, $\A$) and,
    for every $p \in [1,2)$ and every $\varepsilon > 0$,
    a constant $C(\varepsilon) \geq 1$ (depending on $N$, $M$, $\A$, $p$ and $\varepsilon$) such that the following holds.
    If $(u,K)$ is a topological quasiminimizer in $\Omega$,
    then for all $x \in K$ and $r > 0$ with $B(x,r) \subset \Omega$ and $h(r) \leq \varepsilon_0$,
    there exists $y \in K \cap B(x,r/2)$ and $t \in [C(\varepsilon)^{-1} r,r/2]$ such that
    \begin{equation*}
        \beta(y,t) + \omega_p(y,t) \leq \varepsilon.
    \end{equation*}
\end{corollary}
\begin{proof}
    This is standard.
    Fix $\varepsilon_0$ smaller than in Corollary \ref{cor_carleson} and Proposition \ref{prop_flat}.
    Let $\varepsilon_1 > 0$ be a small constant to be chosen later.
    Assuming that $h(r) \leq \varepsilon_0$, Proposition \ref{prop_flat} yields a point $y \in K \cap B(x,r/2)$ and a radius $t \in [C(\varepsilon_1)^{-1}r,r/2]$ such that $\beta(y,t) \leq \varepsilon_1$.
    Applying Corollary \ref{cor_carleson} in $B(y,t)$, we then obtain a point $z \in K \cap B(y,t/2)$ and a radius $s \in [C(\varepsilon)^{-1}t,t/2]$ such that $\omega_p(z,s) \leq \varepsilon/2$.
    Since $s \geq C(\varepsilon)^{-1} t$, we still control $\beta(z,s) \leq C(\varepsilon) \beta(y,t) \leq C(\varepsilon) \varepsilon_1$.
    Finally, we choose $\varepsilon_1$ (depending on $N$, $M$, $\A$, $p$ and $\varepsilon$) so that $\beta(z,s) \leq \varepsilon/2$.
    This choice gives
    \begin{equation*}
        \beta(z,s) + \omega_p(z,s) \leq \varepsilon.
    \end{equation*}
\end{proof}

\section{Fine estimate on holes through a projection}\label{section_jump}

We define a ``normalized jump'' in the spirit of \cite{DavidBOOK}.
Let $(u,K)$ be a pair in a ball $B(x,r) \subset \R^N$ such that $\beta(x,r) \leq 1/2$.
Let $\nu_0$ be a unit normal to a hyperplane realizing the infimum in the definition of $\beta(x,r)$.
Let $a_1, a_2$ be rigid motions approximating $u$ in the balls $D_1 = B(x_1,r/8)$ and $D_2 = B(x_2,r/8)$ centered at $x_1 = x + (3r/4) \nu_0$ and $x_2 = x - (3r/4) \nu_0$.
Each takes the form
\begin{equation*}
    a_i(y) = b_i + A_i(y - x),
\end{equation*}
where $b_i \in \R^N$ and $A_i \in \R^{N \times N}_{\mathrm{skew}}$ are given by
\begin{equation*}
    A_i = \fint_{D_i} \frac{\nabla u(y) - \nabla u(y)^T}{2} \dd{y}, \quad b_i = \fint_{D_i} u(y) \dd{y} + A_i(x - x_i).
\end{equation*}
The \emph{normalized jump} of $u$ in $B(x,r)$ is then
\begin{equation*}
    J(x,r) := \frac{\abs{b_1 - b_2} + r \abs{A_1 - A_2}}{\sqrt{r}}.
\end{equation*}
This quantity is invariant under rescaling to the unit ball; see Remark \ref{rmk_scaling}.

Here is the main result of this section.

\begin{proposition}\label{prop_projection}
    There exist $\varepsilon_0 > 0$ (which depends on $N$, $M$, $\A$) and, for all $p \in (2(N-1)/N,2]$ and all $\varepsilon > 0$, a constant $\varepsilon_1 > 0$ (which depends on $N$, $M$, $\A$, $p$ and $\varepsilon$) such that the following holds.
    Let $(u,K)$ be a topological quasiminimizer with any gauge $h$ in $\Omega$.
    For every $x \in K$ and $r > 0$ with $B(x,r) \subset \Omega$ and 
    \begin{equation*}
        \beta(x,r) + \omega_p(x,r) \leq \varepsilon_1 \quad \text{and} \quad h(r) \leq \varepsilon_0,
    \end{equation*}
    there exists a hyperplane $V \in \Gr(N-1)$ such that, for all $W \in \Gr(N-1)$ with $\mathrm{dist}(V,W) \leq \varepsilon_1$,
    \begin{equation}\label{eq_VW}
        \H^{N-1}\bigl(W \cap B(x,r) \setminus p_W(K \cap B(x,r))\bigr) \leq \varepsilon r^{N-1}.
    \end{equation}
\end{proposition}

For this purpose, we will use \cite[Lemma 4.3]{LL3} on the initialization of the jump. 
Note that the assumption on $D_1, D_2$ is only needed for \emph{topological} quasiminimizers.

\begin{lemma}[Initialization of the jump]\label{lem_jump}
    There exist a constant $\varepsilon_0 > 0$ (depending on $N$, $M$, $\A$) and,
    for every $p \in (2(N-1)/N,2]$, a constant $\eta_0 \geq 1$ (depending on $N$, $M$, $\A$, $p$) such that the following holds.
    Let $(u,K)$ be a topological quasiminimizer with any gauge $h$ in $\Omega$.
    If $x \in K$ and $r > 0$ are such that $B(x,r) \subset \Omega$,
    \begin{equation*}
        \beta(x,r) + \omega_p(x,r) \leq \eta_0, \quad h(r) \leq \varepsilon_0
    \end{equation*}
    and if the balls $D_1, D_2$ (defined at the beginning of Section \ref{section_jump}) lie in the same connected component of $\Omega \setminus K$, then
    \begin{equation*}
        J(x,r) \geq \eta_0.
    \end{equation*}
\end{lemma}

We will also use Lemma 4.4 from \cite{LL3}, which we recall below.

\begin{lemma}[Holes through slicing]\label{lem_slicing}
    Let $(u,K)$ be a pair in $B(0,1)$, and assume that $\beta(0,1) \leq \varepsilon$ for some $\varepsilon \in (0,1/4)$.
    Let $\nu_0$ be a unit normal to a hyperplane realizing the infimum in the definition of $\beta(0,1)$.
    Then, for all unit vectors $\nu \in \mathbf{S}^{N-1}$ with $\abs{\nu - \nu_0} \leq \varepsilon$, setting $V =
    \nu^\perp$, we have
    \begin{equation}\label{eq_slicing0}
        J(\nu) \H^{N-1}(S_{\varepsilon}(V))^2 \leq C \varepsilon^{-1} \int_{B(0,1)} \abs{e(u)} \dd{x},
    \end{equation}
    where $C \geq 1$ is a universal constant,
    \begin{equation*}
        S_{\varepsilon}(V) := V \cap B(0,(1-4\varepsilon)) \setminus p_V(K \cap B(0,1))
    \end{equation*}
    is the set of holes through slicing in the direction $\nu$, and
    \begin{equation*}
        J(\nu) := \abs{(b_1 - b_2) \cdot \nu} + \abs{(A_1 - A_2) \nu}
    \end{equation*}
    is the component of the jump in the direction $\nu$, with $b_i, A_i$ as defined at the beginning of Section \ref{section_jump}.
\end{lemma}

Note that the right-hand side of (\ref{eq_slicing0}) deteriorates as $\varepsilon \to 0$ because of the factor
$\varepsilon^{-1}$, while we gain more information on the left-hand side since $S_{\varepsilon}(V)$ gets larger. In order to deduce
Proposition \ref{prop_projection}, we will fix $\varepsilon > 0$ sufficiently small in Lemma \ref{lem_slicing} so that
$\H^{N-1}(S_{\varepsilon}(V))$ captures most of $\H^{N-1}\bigl(V \cap B(0,1) \setminus p_V(K \cap B(0,1))\bigr)$, and then assume
$\int_{B(0,1)} \abs{e(u)} \dd{x}$ small enough to compensate for the factor $\varepsilon^{-1}$.

\begin{proof}[Proof of Proposition \ref{prop_projection}]
    To simplify notation, we may assume that $B(x,r) = B(0,1)$.
    Let $p \in (2(N-1)/N,2]$ and $\varepsilon > 0$.
    Let $(u,K)$ be a topological quasiminimizer in $B(0,1)$ such that
    \begin{equation}\label{eq_obh}
        \beta(0,1) \leq \varepsilon_2,\quad \omega_p(0,1) \leq \varepsilon_1 \quad \text{and} \quad h(1) \leq \varepsilon_0,
    \end{equation}
    where $\varepsilon_1, \varepsilon_2 \in (0,1/4)$ are constants to be chosen small enough 
    (depending on $N$, $M$, $\A$, $p$ and $\varepsilon$) and $\varepsilon_0$ is the constant of Lemma \ref{lem_jump}.
    We denote by $\nu_0$ a unit normal to a hyperplane realizing the infimum in the definition of $\beta(0,1)$,
    and by $D_1, D_2$ the balls defined at the beginning of Section \ref{section_jump}.
    The letter $C$ denotes a generic constant $\geq 1$ which depends on $N$, $M$, $\A$ and $p$.

    If $D_1$ and $D_2$ lie in distinct components of $B(0,1) \setminus K$, property (\ref{eq_VW}) follows directly.
    Indeed, recall that $D_i = B(x_i,1/8)$, where $x_1 = (3/4) \nu_0$ and $x_2 = -(3/4)\nu_0$, and that 
    \begin{equation*}
        K \cap B(0,1) \subset \set{x \in B(0,1) | \abs{x \cdot \nu_0} \leq \varepsilon_2}.
    \end{equation*}
    For $\nu \in \mathbf{S}^{N-1}$ with $\abs{\nu - \nu_0} \leq \varepsilon_2$, we thus have
    \begin{equation*}
        K \cap B(0,1) \subset \set{x \in B(0,1) | \abs{x \cdot \nu} \leq 2\varepsilon_2}.
    \end{equation*}
    Moreover, $\abs{x_i \cdot (\nu - \nu_0)} \leq \varepsilon_2$ for $i=1,2$, so that $x_1 \cdot \nu > 2 \varepsilon_2$ and $x_2
    \cdot \nu < -2 \varepsilon_2$.
    Let $W = \nu^\perp$. For every $x \in W \cap B(0,(1 - 4\varepsilon_2))$, the segment $x +
    [-2\varepsilon_2,2\varepsilon_2] \nu$ intersects $K$; otherwise it could be used to connect the centers of $D_1$ and $D_2$.
    Hence, the projection $p_W(K \cap B(0,1))$ contains $W \cap B(0,1-4\varepsilon_2)$ and we deduce
    \begin{equation*}
        \H^{N-1}\bigl(W \cap B(0,1) \setminus p_W(K \cap B(0,1))\bigr) \leq C \varepsilon_2.
    \end{equation*}
    We finally choose $\varepsilon_2$ so that the right-hand side is less than $\varepsilon$.

    Suppose now that $D_1$ and $D_2$ are not separated by $K$ in $B(0,1)$.
    Assume that $\varepsilon_1, \varepsilon_2$ are at most $\eta_0/2$, where $\eta_0$ is the constant of Lemma \ref{lem_jump}.
    Lemma \ref{lem_jump} then yields
    \begin{equation}\label{eq_init_jump}
        \abs{b_1 - b_2} + \abs{A_1 - A_2} \geq C^{-1}.
    \end{equation}
    Lemma \ref{lem_slicing} further gives that, for every unit vector $\nu \in \mathbf{S}^{N-1}$
    with $\abs{\nu - \nu_0} \leq \varepsilon_2$, we have
    \begin{equation}\label{eq_slicing}
        J(\nu) \H^{N-1}(S_{\varepsilon_2}(W))^2 \leq C \varepsilon_2^{-1} \omega_1(0,1)^{1/2},
    \end{equation}
    where $W = \nu^\perp$.

    We now construct a direction $\tau_0 \in \mathbf{S}^{N-1}$ such that $J(\nu)$ admits a uniform lower bound for all $\nu$ in a neighborhood of $\tau_0$.
    For $\delta > 0$, the set of vectors $\nu \in \mathbf{S}^{N-1}$ such that
    \begin{equation}\label{eq_nubi}
        \abs{(b_1 - b_2) \cdot \nu} < \delta \abs{b_1 - b_2}
    \end{equation}
    has measure at most $C_0 \delta$, for some universal constant $C_0 \geq 1$, since it is the intersection of the unit sphere with a $\delta$-neighborhood of a hyperplane.
    A similar argument shows that the set of $\nu \in \mathbf{S}^{N-1}$ with
    \begin{equation*}
        \abs{(A_1 - A_2) \nu} < \delta \abs{A_1 - A_2}
    \end{equation*}
    has measure $\leq N C_0 \delta$. Indeed, let $(c_1,\ldots,c_N)$ denote the columns of $(A_1 - A_2)^T$;
    the condition above reads
    \begin{equation*}
        \sum_k \abs{\nu \cdot c_k}^2 < \delta \sum_k \abs{c_k}^2,
    \end{equation*}
    so there exists $k = 1,\ldots,N$ satisfying $\abs{\nu \cdot c_k} < \delta \abs{c_k}$, which is the same form as (\ref{eq_nubi}).
    We can thus fix $\delta > 0$ small enough (depending only on $N$ and $\varepsilon_2$) so that
    some $\tau_0 \in \mathbf{S}^{N-1}$ with $\abs{\nu_0 - \tau_0} \leq \varepsilon_2/2$ satisfies
    \begin{equation*}
        \abs{(b_1 - b_2) \cdot \tau_0} \geq \delta \abs{b_1 - b_2}
    \end{equation*}
    as well as
    \begin{equation*}
        \abs{(A_1 - A_2) \tau_0} \geq \delta \abs{A_1 - A_2}.
    \end{equation*}
    This gives $J(\tau_0) \geq \delta \bigl(\abs{b_1 - b_2} + \abs{A_1 - A_2}\bigr)$.
    (The constant $\delta$ is now fixed for the remainder of the proof.)
    Observe that the map
    \begin{equation*}
        \nu \mapsto \frac{\abs{(b_1 - b_2) \cdot \nu} + \abs{(A_1 - A_2) \nu}}{\abs{b_1 - b_2} + \abs{A_1 - A_2}}
    \end{equation*}
    is Lipschitz with a universal constant $L$. Then, for every $\nu \in \mathbf{S}^{N-1}$ with $\abs{\nu - \tau_0} \leq L^{-1} \delta/2$, we have 
    \begin{equation*}
        J(\nu) \geq J(\tau_0) - \frac{\delta}{2} \bigl(\abs{b_1 - b_2} + \abs{A_1 - A_2}\bigr) \geq \frac{\delta}{2} \bigl(\abs{b_1 - b_2} + \abs{A_1 - A_2}\bigr) 
    \end{equation*}
    and in particular $J(\nu) \geq C^{-1} \delta$ by (\ref{eq_init_jump}).

    Now, (\ref{eq_slicing}) implies that for all $\nu \in \mathbf{S}^{N-1}$ with $\abs{\nu - \tau_0} \leq
    \min(\varepsilon_2/2, L^{-1}\delta/2)$,
    \begin{equation*}
        \H^{N-1}(S_{\varepsilon_2}(W))^2 \leq C \delta^{-1} \varepsilon_2^{-1} \omega_1(0,1)^{1/2},
    \end{equation*}
    where $W = \nu^\perp$.
    We readily observe
    $\H^{N-1}\bigl(W \cap B(0,1) \setminus B(0,(1-4\varepsilon_2))\bigr) \leq C \varepsilon_2$,
    so that in fact
    \begin{equation*}
        \H^{N-1}\bigl(W \cap B(0,1) \setminus p_W(K \cap B(0,1))\bigr) \leq C \sqrt{\delta^{-1} \varepsilon_2^{-1} \omega_1(0,1)^{1/2}} + C \varepsilon_2.
    \end{equation*}
    We finally fix $\varepsilon_2$ such that $C \varepsilon_2 \leq \varepsilon/2$ and, using $\omega_1 \leq C \omega_p$ by Hölder's
    inequality, we then choose $\varepsilon_1$ so small that
    \begin{equation*}
        \H^{N-1}\bigl(W \cap B(0,1) \setminus p_W(K \cap B(0,1))\bigr) \leq \varepsilon.
    \end{equation*}
    This completes the proof of Proposition \ref{prop_projection}.
\end{proof}

\section{Plenty of big projections}

The following two results are immediate consequences of Corollary \ref{cor_nice} and Proposition \ref{prop_projection}, and their proofs are omitted.
The first one is the \emph{uniform concentration} property; in this way, our argument provides a variant of the proof in \cite{LL3} and leads to a stronger result, as it extends to quasiminimizers.

\begin{corollary}[Uniform concentration]
    There exist $\varepsilon_0 > 0$ (which depends on $N$, $M$, $\A$) and, for all $\varepsilon > 0$, a constant $C(\varepsilon) \geq 1$ (which depends on $N$, $M$, $\A$, $\varepsilon$) such that the following holds.
    If $(u,K)$ is a topological quasiminimizer in a domain $\Omega$, then for all $x \in K$ and all $r > 0$ with $B(x,r)
    \subset \Omega$ and $h(r) \leq \varepsilon_0$, there exists $y \in K \cap B(x,r/2)$ and $t \in [C(\varepsilon)^{-1}r,r/2]$ such that
    \begin{equation*}
        \H^{N-1}(K \cap B(y,t)) \geq (1 - \varepsilon) \omega_{N-1} t^{N-1},
    \end{equation*}
    where $\omega_{N-1}$ is the measure of the $(N-1)$-dimensional unit disk.
\end{corollary}

The second one is the \emph{plenty of big projections} property, which we use to establish the uniform rectifiability of fractures.

\begin{corollary}[Plenty of big projections]\label{cor_pbp}
    There exist constants $\varepsilon_0 > 0$, $\delta > 0$ and $C \geq 1$ (which depend on $N$, $M$, $\A$) such that the following holds.
    If $(u,K)$ is a topological quasiminimizer in a domain $\Omega$, then for all $x \in K$ and all $r > 0$ with $B(x,r) \subset \Omega$ and $h(r) \leq \varepsilon_0$, there exists a hyperplane $V \in \Gr(N-1)$ such that for all $W \in \Gr(N-1)$ with $\mathrm{dist}(V,W) \leq \delta$,
    \begin{equation*}
        \H^{N-1}\bigl(p_W(K \cap B(x,r))\bigr) \geq C^{-1} r^{N-1}
    \end{equation*}
\end{corollary}

We finally deduce our main result.

\begin{corollary}[Uniform rectifiability]
    There exist constants $\varepsilon_0 > 0$ and $C \geq 1$ (which depend on $N$, $M$, $\A$) such that the following holds.
    If $(u,K)$ is a topological quasiminimizer in a domain $\Omega$, then for all $x \in K$ and all $r > 0$ with $B(x,2r) \subset \Omega$ and $h(2r) \leq \varepsilon_0$,
    \begin{equation*}
        \text{$K \cap B(x,r)$ is contained in a uniformly rectifiable set $E$ (with BPLG) of constant $C$.}
    \end{equation*}
\end{corollary}
\begin{proof}
    This is essentially an application of \cite[Theorem 1.6]{Orponen}, which states that plenty of big projections (PBP) implies big pieces of Lipschitz graphs (BPLG).
    A minor adaptation is required since our PBP property (Corollary \ref{cor_pbp}) is local, whereas \cite[Theorem 1.6]{Orponen} is global.
    Fix $x \in K$ and $r > 0$ such that $B(x,2r) \subset \Omega$ and $h(2r) \leq \varepsilon_0$, where $\varepsilon_0$ is smaller than in Proposition \ref{prop_AR} and Corollary \ref{cor_pbp}.
    Consider any hyperplane $P$ passing through $x$ and set
    \begin{equation*}
        E = \bigl(K \cap B(x,r)\bigr) \cup \partial B(x,r) \cup \bigl(P \setminus B(x,r)\bigr).
    \end{equation*}
    The set $E$ is closed and Ahlfors-regular (we omit the proof). Let us verify that $E$ has plenty of big projections.
    The key point is that hyperplanes and spheres have plenty of big projections: there exist universal constants $\delta > 0$ and $C \geq 1$ such that for all $y \in P$ and $\rho > 0$,
    \begin{equation}\label{eq_plane_bp}
        \exists V \in \Gr(N-1),\; \forall W \in B(V,\delta),\; \H^{N-1}\bigl(p_W(P \cap B(y,\rho))\bigr) \geq C^{-1} \rho^{N-1}
    \end{equation}
    and for all $y \in \partial B(x,r)$ and $0 < \rho < 2r$,
    \begin{equation}\label{eq_sphere_bp}
        \exists V \in \Gr(N-1),\; \forall W \in B(V,\delta),\; \H^{N-1}\bigl(p_W(\partial B(x,r) \cap B(y,\rho))\bigr) \geq C^{-1} \rho^{N-1}.
    \end{equation}

    Now, fix $y \in E$ and $\rho > 0$.
    Our first case is when $B(y,\rho/2) \subset B(x,r)$.
    Then necessarily $y \in K$, and we may apply Corollary \ref{cor_pbp} in $B(y,\rho/2)$.
    Since $K \cap B(y,\rho/2) \subset E \cap B(y,\rho)$, it follows that $E$ has plenty of big projections in $B(y,\rho)$.

    Our second case is when $B(y,3\rho/4)$ contains a point $z \in P \setminus B(x,r+\rho/100)$.
    We apply (\ref{eq_plane_bp}) in $B(z,\rho/100)$ and, because $P \cap B(z,\rho/100) \subset E \cap B(y,\rho)$, we again obtain that $E$ has plenty of big projections in $B(y,\rho)$.

    We are left with the case where $B(y,\rho/2) \not\subset B(x,r)$ and $B(y,3\rho/4)$ does not meet $P \setminus B(x,r+\rho/100)$.
    From the second condition, we immediately obtain $\abs{x - y} < r + \rho/100$.
    In particular $B(y,\rho/2) \cap B(x,r) \ne \emptyset$, 
    and since $B(y,\rho/2) \not\subset B(x,r)$,
    there exists a point $z \in B(y,\rho/2) \cap \partial B(x,r)$.
    We also check that $\rho < 4 r$: if this were not the case, then, using again $\abs{x - y} < r + \rho/100$, we would have $$\overline{B}(x,r+\rho/100) \subset B(y,2r+\rho/50) \subset B(y,3\rho/4)$$ and therefore any point of $P \cap \partial B(x,r+\rho/100)$
    would belong to both $B(y,3\rho/4)$ and $P \setminus B(x,r+\rho/100)$.
    Since $\rho/2 < 2r$, we may apply (\ref{eq_sphere_bp}) in $B(z,\rho/2)$.
    As $\partial B(x,r) \cap B(z,\rho/2) \subset E \cap B(y,\rho)$, this shows once more that $E$ has plenty of big projections in $B(y,\rho)$.
\end{proof}

%

\appendix
\section{Appendix}

\begin{lemma}\label{lem_rigid}
    For all real numbers $p \geq 1$, for all affine maps $a \colon \R^{n} \to \R^{m}$, for all balls $B \subset \R^{n}$ and for all Borel sets $E \subset B$ such that $\abs{E} > 0$, we have
    \begin{equation}\label{eq_rigid}
        \biggl(\fint_E \abs{a(x)}^p \dd{x}\biggr)^{1/p} \geq C^{-1} \norm{a}_{L^{\infty}(B)} \biggl(\frac{\abs{E}}{\abs{B}}\biggr),
    \end{equation}
    where $C \geq 1$ depends on $n$, $m$, $p$.
\end{lemma}
\begin{proof}
    By scale invariance of the inequality, we may assume that $B = B(0,1)$.
    Writing $a_1,\ldots,a_m$ for the coordinates of $a$, we have for all $x \in \R^n$, 
    \begin{equation*}
        C^{-1} \sum_{i=1}^m \abs{a_i(x)}^p \leq \abs{a(x)}^p \leq C \sum_{i=1}^m \abs{a_i(x)}^p
    \end{equation*}
    so it suffices to prove (\ref{eq_rigid}) for each coordinate of $a$.
    Thus we may restrict to the case where $a$ is an affine map $a \colon \R^n \to \R$ of the form $a(x) = c + v \cdot x$, where $c \in \R$, $v \in \R^n$.
    We can also assume $(c,v) \ne 0$, otherwise (\ref{eq_rigid}) is trivial.

    The main point is to show that there exists a universal constant $C_0 \geq 1$ such that for all $t > 0$,
    \begin{equation}\label{eqE}
        \abs{\set{x \in B | \abs{a(x)} \leq t}} \leq \frac{C_0 t}{\abs{c} + \abs{v}}.
    \end{equation}
    If $\abs{c} \geq 2 \abs{v}$, then $\abs{a(x)} \geq \abs{c}/2$ in $B$, and we distinguish two cases.
    Either $\abs{c}/2 > t$, in which case
    \begin{equation*}
        \abs{\set{x \in B | \abs{a(x)} \leq t}} = 0
    \end{equation*}
    or $\abs{c}/2 \leq t$, and then we have the trivial bound
    \begin{equation*}
        \abs{\set{x \in B | \abs{a(x)} \leq t}} \leq \abs{B} \leq \frac{2\abs{B}t}{\abs{c}} \leq \frac{3\abs{B}t}{\abs{c} + \abs{v}}.
    \end{equation*}
    Next, if $\abs{c} < 2 \abs{v}$, then in particular $v \ne 0$, and the condition $\abs{a(x)} \leq t$ describes a $(\abs{v}^{-1} t)$-neighborhood of a hyperplane, so
    \begin{equation*}
        \abs{\set{x \in B | \abs{a(x)} \leq t}} \leq \frac{C t}{\abs{v}} \leq \frac{3 C t}{\abs{c} + \abs{v}},
    \end{equation*}
    for some universal constant $C \geq 1$.
    This proves (\ref{eqE}).

    For $t > 0$, we then estimate
    \begin{align*}
        \int_E \abs{a(x)}^p \dd{x} &\geq t^p \abs{\set{x \in E | \abs{a(x)} > t}}\\
                                   &\geq t^p \Bigl(\abs{E} - \abs{\set{x \in B | \abs{a(x)} \leq t}}\Bigr)\\
                                   &\geq t^p \Bigl(\abs{E} - \frac{C_0 t}{\abs{c} + \abs{v}}\Bigr)
    \end{align*}
    and the desired inequality follows by choosing $t$ so that
    \begin{equation*}
        \abs{E} - \frac{C_0 t}{\abs{c} + \abs{v}} = \frac{\abs{E}}{2}.
    \end{equation*}
\end{proof}

\bibliographystyle{plain}
\bibliography{biblio_griffith}

\medskip
\noindent
\textsc{Université de Lorraine, CNRS, IECL, F-54000 Nancy, France}\\
\textit{Email address:} \url{camille.labourie@univ-lorraine.fr}

\end{document}